\documentclass[12pt]{amsart}
\usepackage[utf8]{inputenc}

\usepackage[margin=1.3in]{geometry}

\usepackage{amsmath}
\usepackage{amsfonts}
\usepackage{amssymb}
\usepackage{amsthm}
\usepackage{mathtools}
\usepackage{caption}
\usepackage{subcaption}
\usepackage{bbm}
\usepackage[export]{adjustbox}

\usepackage{stmaryrd}

\usepackage[all]{xy}

\usepackage{tikz-cd}
\usetikzlibrary{matrix}
\usepackage{graphicx} 
\usepackage{float}

\usepackage{epstopdf}

\usepackage{overpic}

\usepackage[linktocpage]{hyperref}
\hypersetup{
    colorlinks=true,
    linkcolor=blue,
    citecolor=blue,      
    urlcolor=blue,
}

\usepackage{color}
\definecolor{note}{rgb}{0,0,1}  

\newtheorem{theorem}{Theorem}
\newtheorem{definition}[theorem]{Definition}
\newtheorem{proposition}[theorem]{Proposition}
\newtheorem{lemma}[theorem]{Lemma}

\newtheorem{corollary}[theorem]{Corollary}

\newtheorem{remark}[theorem]{Remark}

\numberwithin{equation}{section}
\numberwithin{theorem}{section}

\usepackage{enumitem}

\usepackage{todonotes}

\newcommand{\op}{\operatorname}

\newcommand{\R}{\mathbb{R}}

\newcommand{\be}{\begin{enumerate}}
\newcommand{\ee}{\end{enumerate}}

\usepackage[english]{babel}
\usepackage{csquotes}

\usepackage{hyphenat}

\usepackage[backend=biber,style=alphabetic,maxalphanames=4,maxnames=4]{biblatex}

\renewbibmacro{in:}{}

\DeclareDelimFormat[bib,biblist]{nametitledelim}{\addcomma\space}

\DeclareFieldFormat*{title}{\mkbibitalic{#1}\addcomma}
\DeclareFieldFormat*{journaltitle}{#1}
\DeclareFieldFormat*{volume}{\mkbibbold{#1}}
\DeclareFieldFormat{pages}{#1}
\DeclareFieldFormat[misc]{date}{preprint {#1}}
\DeclareFieldFormat{mr}{%
  MR\addcolon\space
  \ifhyperref
    {\href{http://www.ams.org/mathscinet-getitem?mr=MR#1}{\nolinkurl{#1}}}
    {\nolinkurl{#1}}}
    
\AtEveryBibitem{
  \clearfield{url}
  \clearfield{issn}
  \clearfield{isbn}
  \clearfield{eprintclass}
}
\AtEveryBibitem{\ifentrytype{book}{\clearfield{pages}}{}}

\bibliography{trapezoid_flat}

\usepackage{fancyhdr}
\title{Trapezoidality of flat arrangement polynomials}

\author{Yuan Gao}
\address{School of Mathematics, Nanjing University, Nanjing, Jiangsu, 210093, China}
\email{yuangao@nju.edu.cn} \urladdr{}

\author{Tianyu Yuan}
\address{School of Mathematical Sciences, Eastern Institute of Technology, Ningbo, Zhejiang, 315200, China}
\email{tyyuan@eitech.edu.cn} \urladdr{}

\date{\today}

\keywords{flat vector arrangements, barycentric fibers, h-polynomials, Eulerian digraphs, Alexander polynomial}

\subjclass[2020]{Primary 52B40; Secondary 52C40, 05C31, 57K14.}

\begin{document}

\begin{abstract}
    Fox's conjecture predicts that the absolute values of the coefficients of the Alexander polynomial of an alternating link form a trapezoidal sequence. Kálmán, Mészáros, and Postnikov gave a new proof of trapezoidality for special alternating links using a polynomial associated with flat vector arrangements. We prove that the flat arrangement polynomial has trapezoidal coefficients for every full-dimensional flat vector arrangement, answering the corresponding open question. Our proof gives a geometric interpretation of this polynomial in terms of convex polytopes. As an application, we prove that the Murasugi--Stoimenow polynomial of every connected Eulerian digraph has trapezoidal coefficients. 
\end{abstract}

\maketitle

\tableofcontents

\section{Introduction}

Fox's conjecture asserts that the coefficients of the Alexander polynomial $\Delta_L(-t)$ form a trapezoidal sequence for every alternating link $L$ \cite{Fox62}.
The conjecture remains open in general, although Hafner, Mészáros, and Vidinas proved the stronger conclusion of log-concavity for special alternating links \cite[Theorem~1.2]{HMV24}.

For alternating links, coefficient questions of this kind have long admitted graph-theoretic models. Crowell expressed the Alexander polynomial as a weighted sum over arborescences of a directed graph associated with an alternating diagram \cite{Crowell59}. In the special alternating setting, Murasugi and Stoimenow extended this graph-enumerative picture to the spanning-tree polynomial $P_D(t)$ of an Eulerian digraph; for alternating dimaps, $P_D(t)$ recovers the Alexander polynomial up to the standard sign and monomial normalization \cite{MS03}. Hafner, Mészáros, and Vidinas subsequently placed this Eulerian-digraph polynomial in a broader combinatorial and discrete-geometric framework \cite{HMVGeneralized}.

A parallel line of work gives polyhedral models for the same circle of graph enumerators.
Kálmán and Murakami identified the top of the HOMFLYPT polynomial of a special alternating link with the $h$-vector of a triangulation of a root polytope and with a parking-function enumerator \cite{KM17}. Kálmán and Postnikov related the interior polynomial of a hypergraph to the Ehrhart polynomial of the associated root polytope, or equivalently to the $h$-vectors or its triangulations \cite{KP17}. Activity interpretations of Ehrhart $h^*$-polynomials for graph and oriented-matroid root polytopes were developed further in \cite{KT23,Tothmeresz24}. In the Eulerian-digraph setting, Hafner, Mészáros, and Vidinas expressed $P_D(t)$ through volumes of root polytopes of oriented co-Eulerian matroids \cite[Theorem~1.7]{HMVGeneralized}.

Within this circle of ideas, Kálmán, Mészáros, and Postnikov attribute to Li and Postnikov the determinant-weighted external semi-activity polynomial $f_A(t)$ of a flat vector arrangement, together with its independence of an auxiliary generic vector; see \cite{LP13} and \cite[Definition~3.3 and Theorem~3.4]{KMP}.
They subsequently studied $f_A(t)$ in the Alexander-polynomial setting and asked whether its coefficient sequence is trapezoidal for every flat arrangement \cite[p.~2]{KMP}. We answer this question for arbitrary flat arrangements by realizing external semi-activity on a basis as an edge-orientation statistic on a family of convex polytopes.

The geometric mechanism comes from point fibers of the projection of a standard simplex.
Denote $V=\R^d$. Let $A=(a_1,\dots,a_N)$ be a $d$-by-$N$ matrix, where $a_1,\dots,a_N$ are column vectors in $V$. We call this vector arrangement \emph{full-dimensional} if its columns span $V$, equivalently if $\op{rank}A=d$, and assume this throughout. Suppose a linear functional $\varphi\in V^*$ satisfies $\varphi(a_i)=1$ for each $i$. All vectors $a_i$ then lie in the affine hyperplane $H=\varphi^{-1}(1)$. If $x\in H$ and $A\lambda=x$ with $\lambda\in \R^N_{\geq0}$, applying $\varphi$ gives $\sum_i\lambda_i=1$. Consequently,
\begin{equation*}
  Q_x=\{\lambda\in \R^N_{\geq0}\,\,|\,\,A\lambda=x\}
\end{equation*}
is exactly the fiber over $x$ of the projection
\begin{equation*}
  \pi\colon \Delta^{N-1} \longrightarrow C\coloneqq \op{conv}(\{a_1,\dots,a_N\}),\quad \lambda\mapsto A\lambda.
\end{equation*}
For $x \in C$, we call
\[
  Q_x=\pi^{-1}(x)
\]
the \emph{barycentric fiber of $A$ over $x$}, since its points are
precisely the labeled nonnegative barycentric representations of $x$
by the columns of $A$.
In Section~\ref{sec-barycentric-fibers}, we specify a full-measure
subset
\[
  C_{\mathrm{generic}}\subset C
\]
and call $Q_x$ a \emph{generic barycentric fiber} when
$x\in C_{\mathrm{generic}}$.
These barycentric fibers are classical point fibers of a polytope projection. Billera and Sturmfels package a family of fibers into a single fiber polytope by Minkowski integration \cite{BS92}; our construction instead integrates the ordinary $h$-polynomials of the individual generic fibers.

For $x\in C_{\mathrm{generic}}$, the barycentric fiber $Q_x$ is a simple polytope of dimension $N-d$. The vertices of $Q_x$ are supported on the labeled bases $B\subset \{1,\dots,N\}$ for which the basis simplex
\begin{equation*}
  C_B=\op{conv}(\{a_b\,\,|\,\,b\in B\})
\end{equation*}
contains $x$. At the vertex supported on $B$, its edge rays correspond to vectors $c^{B,j}\in\R^N$ indexed by $j\notin B$. An additional choice of a generic vector $\rho\in\R^N$ orients the edge rays, and the number of rays along which $\langle\rho,-\rangle$ increases equals the external semi-activity $\op{ext}_\rho(B)$.
This matches the classical formula for a simple polytope $Q$ and a generic linear functional:
\begin{equation*}
  h(Q;t)\coloneqq \sum_i f_i(Q)(t-1)^i = \sum_{v\in \op{Vert}(Q)}t^{w(v)},
\end{equation*}
where $f_i(Q)$ is the number of $i$-faces of $Q$ and $w(v)$ is the number of edge rays from $v$ along which the linear functional increases. 
With the full-measure generic locus $C_{\mathrm{generic}}\subset C$ and the natural measure $\mu$ introduced in Section~\ref{sec-integration}, we obtain the key identity
\begin{equation}
  \label{eq-intro-barycentric-formula}
  f_A(t) = d!\int_{C_{\mathrm{generic}}} h(Q_x;t)\,d\mu(x).
\end{equation}
Thus $f_A(t)$ is a positive integral of the ordinary $h$-polynomials of the generic barycentric fibers. The determinant weights arise from the affine volumes of the basis simplices in the base $C$.

\begin{definition}
  \label{def-trapezoidal}
  Let $a_0,\dots,a_r$ be a strictly positive palindromic sequence. Put $m=\lfloor r/2\rfloor$ and $\delta_i=a_i-a_{i-1}$ for $1\leq i\leq m$.
  We say that the sequence is \emph{ trapezoidal} if
  \begin{equation}
    \delta_i\geq 0,\quad \delta_i=0 \Longrightarrow \delta_j=0\text{  for  } i\leq j\leq m.
  \end{equation}
  For a palindromic sequence, this is equivalent to the usual strictly increasing $\to$ constant $\to$ strictly decreasing formulation.
\end{definition}

The Dehn-Sommerville symmetry and the $M$-sequence part of the $g$-theorem \cite[Theorem~8.35]{Zie} imply that the coefficient sequence of $h(Q_x;t)$ is trapezoidal. In this way, the $g$-theorem governs the
coefficient shape of $h(Q_x;t)$. Positive integration transfers the
resulting symmetry and first-half inequalities to $f_A(t)$, while
the zero-propagation property of $M$-sequences accounts for the
single centered plateau.

\begin{theorem}
  \label{thm-main-intro}
  Let $A$ be a full-dimensional flat vector arrangement. The coefficient sequence of $f_A(t)$ is trapezoidal.
\end{theorem}

As an application of Theorem \ref{thm-main-intro}, we discuss trapezoidality of connected Eulerian digraphs.
Murasugi and Stoimenow proposed the trapezoidal coefficient shape for plane even-valence graphs with alternating orientation and asked how far such properties extend beyond the planar setting \cite[Sections~7.1 and~7.4]{MS03}. Log-concavity is known for alternating dimaps, and Hafner, Mészáros, and Vidinas conjectured it for every Eulerian digraph \cite[Theorem~1.5 and Conjecture~1.6]{HMVGeneralized}. 

Murasugi and Stoimenow associate a polynomial $P_D(t)$ to every connected Eulerian digraph $D$ by counting spanning trees according to the number of edges that must be reversed to orient the tree toward a fixed root; the resulting polynomial is independent of the root \cite[Proposition~1]{MS03}. Kálmán, Mészáros, and Postnikov realize $P_D(t)$ as the flat arrangement polynomial of a totally unimodular cographic matrix \cite[Sections~4.2 and~5.1]{KMP}. Hence Theorem \ref{thm-main-intro} implies

\begin{corollary}
  \label{cor-euleriandiagraph-trapezoidal}
  For every connected Eulerian digraph $D$, the coefficient sequence of $P_D(t)$ is trapezoidal.
\end{corollary}

For alternating dimaps, $P_D(t)$ specializes to the Alexander-polynomial model of special alternating links, while the cographic realization applies uniformly to every connected Eulerian digraph. Thus the same barycentric-fiber formula gives both the all-flat theorem and a trapezoidality theorem for the full Eulerian-digraph family.

The paper is organized as follows.
Section~\ref{sec-flat-arrangements} recalls flat arrangements, fundamental circuits, and external semi-activity.
Section~\ref{sec-barycentric-fibers} studies the geometry of barycentric fibers and proves the basis--vertex and tangent-cone descriptions on the generic locus.
Section~\ref{sec-fiber-h} identifies the ordinary $h$-polynomials
of the generic barycentric fibers with local activity enumerators.
Section~\ref{sec-integration} proves the integral formula and derives trapezoidality from the $g$-theorem.
Section~\ref{sec-eulerian-digraphs} proves the Eulerian-digraph application.

\vskip0.5cm
\noindent \textit{Acknowledgements}. 
The authors thank Yin Tian for introducing this project and helpful discussions.
The authors acknowledge the use of ChatGPT in exploratory discussions during the early development of this work. 
In particular, the discussion suggested applying the classical point-fiber viewpoint for projections of the standard simplex to the flat-arrangement polynomial. This suggestion motivated the barycentric-fiber approach developed here. 

\section{Flat arrangements and external semi-activity}
\label{sec-flat-arrangements}

We review the construction originating in \cite{LP13}, following the exposition in \cite[Section~3]{KMP}. Denote the index set $[N]=\{1,\dots,N\}$. Let $V=\R^d$, $d\geq1$, and
\begin{equation}
   A:\R^N\to V
\end{equation}
be a surjective linear map, i.e., a full-dimensional $d$-by-$N$ matrix. Denote the columns by $a_i=Ae_i$.
We assume that $A$ is \emph{flat}: there is a nonzero $\varphi\in V^*$ such that $\varphi(a_i)=1$ for each $i$. Denote the affine hyperplane $H=\varphi^{-1}(1)$ and the convex hull $C=\op{conv}(\{a_1,\dots,a_N\})$.
Since the columns of $A$ span $V$, their affine span is $H$, so $C$ has dimension $d-1$ in $H$.

A subset $S\subset[N]$ is called a \emph{circuit} of $A$ if the columns
$\{a_i\,\,|\,\,i\in S\}$ are linearly dependent, while the columns indexed by every proper subset of $S$ are linearly independent.
Equivalently, $S$ is minimal by inclusion among the supports of nonzero vectors in $\op{ker}A$.

We say a subset $B\subset[N]$ is a \emph{labeled basis} if $|B|=d$ and $\{a_b\,\,|\,\, b\in B\}$ forms a basis of $V$. Therefore, the convex hull
\begin{equation*}
  C_B\coloneqq \op{conv}(\{a_b,\,b\in B\})
\end{equation*}
is a $(d-1)$-simplex in $H$ and has positive volume in $H$. We write
\begin{equation*}
  \Pi_B\coloneqq \left\{\sum_{b\in B}\lambda_b a_b\,\,\,|\,\, \lambda_b\in[0,1]\right\}
\end{equation*}
for the basis parallelotope.

Given a labeled basis $B$ and $j\notin B$, there is a unique vector $c^{B,j}=(c^{B,j}_1,\dots,c^{B,j}_N)\in \R^N$ such that
\begin{equation*}
  c^{B,j}_j=1,\quad c^{B,j}_B=-A_B^{-1}a_j,\quad c^{B,j}_k=0\text{  for  } k\notin B\sqcup\{j\},
\end{equation*}
where $A_B:\R^B\to V$ is the basis isomorphism. We say $c^{B,j}$ is a \emph{fundamental circuit vector}. A direct calculation gives $c^{B,j}\in \op{ker}A$.
We say a nonzero $c\in\op{ker}A$ is a \emph{circuit vector} if its support is a circuit of $A$.
The support of $c^{B,j}$ is the unique circuit contained in $B\sqcup\{j\}$, so every fundamental circuit vector is a circuit vector.

A vector $\rho\in \R^N$ is called \emph{circuit-generic} if $\langle\rho,c\rangle\neq0$ for every circuit vector $c$. Such vectors exist, since the nongeneric vectors form a finite union of proper hyperplanes in $\mathbb{R}^N$.

Fix a circuit-generic $\rho$. We say a circuit vector $c$ is \emph{oriented} if and only if $\langle\rho,c\rangle>0$. Following \cite[Definition~3.1]{KMP}, an index $j\in [N]\backslash B$ is called \emph{externally semi-active} for $B$ if the $j$-coordinate of the oriented fundamental circuit vector supported in $B\sqcup\{j\}$ is positive.
Let $\op{Ext}_\rho(B)$ be the set of these labels and write $\op{ext}_{\rho}(B)=|\op{Ext}_\rho(B)|$.
The following is then immediate from the definition.

\begin{lemma}
  \label{lemma-sign}
  For each labeled basis $B$ and $j\notin B$, $j\in \op{Ext}_\rho(B)$ if and only if $\langle\rho,c^{B,j}\rangle>0$.
\end{lemma}

The flat arrangement polynomial is then defined as
\begin{equation}
  f_{A,\rho}(t)\coloneqq \sum_{B}t^{\op{ext}_{\rho}(B)}\op{Vol}\Pi_B,
\end{equation}
where the sum is over all labeled bases $B$ of $A$, and $\op{Vol}$ is the Lebesgue measure determined by the standard volume form on $V$. The invariance theorem of Li and Postnikov \cite{LP13}, stated in \cite[Theorem~3.4]{KMP}, says that this polynomial is independent of $\rho$. Proposition \ref{prop-bary-formula} below gives another proof, and we write the common polynomial as $f_A(t)$.

\section{Barycentric fibers}
\label{sec-barycentric-fibers}

In this section, we study the nonnegative representation polytopes associated with the simplex projection determined by $A$, which we will call \emph{barycentric fibers}. For a generic fiber, its vertices correspond to labeled bases $B\subset [N]$, and the increasing edge rays at the vertex indexed by $B$ correspond to $\op{Ext}_{\rho}(B)$.

For $x\in V$, define
\begin{equation}
  \label{eq:barycentric-fiber}
  Q_x\coloneqq \{\lambda\in \R^N_{\geq0}\,|\, A\lambda=x\}.
\end{equation}
Let
\begin{equation*}
  \Delta^{N-1}\coloneqq \{\lambda\in \R^N_{\geq0}\,|\, \sum_i\lambda_i=1\}
\end{equation*}
be the standard $(N-1)$-simplex.
If $x\in H$ and $\lambda\in Q_x$, then $\sum_i\lambda_i=\sum_i \varphi(a_i)\lambda_i=\varphi(A\lambda)=\varphi(x)=1$, hence $x\in C$ and $Q_x$ is the fiber over $x\in C$ of the projection
\begin{equation}
  \label{eq:simplex-projection}
  \pi: \Delta^{N-1}\to C,\quad \pi(\lambda)=A\lambda.
\end{equation}
We use the following generic locus:
\begin{equation}
  \label{eq:generic-locus}
  C_{\mathrm{generic}}\coloneqq C^{\circ}\backslash \bigcup_{\emptyset\neq S\subset[N],\,|S|<d}\op{conv}(\{a_i,\, i\in S\}),
\end{equation}
where $C^\circ$ denotes the relative interior of $C$ in $H$.

\begin{definition}[Barycentric and generic fibers]
  \label{def:barycentric-fiber}
  For $x\in C$, the polytope $Q_x$ in \eqref{eq:barycentric-fiber} is called the \emph{barycentric fiber of $A$ over $x$}.
  Equivalently,
  \[
    Q_x=\pi^{-1}(x)
  \]
  for the simplex projection \eqref{eq:simplex-projection}.
  Its points are precisely the labeled barycentric representations
  \[
    x=\sum_{i=1}^{N}\lambda_i a_i, \qquad \lambda_i\geq 0, \qquad \sum_{i=1}^{N}\lambda_i=1.
  \]
  If $x\in C_{\mathrm{generic}}$, we call $Q_x$ a \emph{generic barycentric fiber}.
\end{definition}

The adjective \emph{generic} refers to the base point $x$. It does not mean that all generic barycentric fibers have the same combinatorial type. Its role is to ensure that every vertex of $Q_x$ is supported on a labeled basis and that $Q_x$ is simple.

It is easy to check that
\begin{lemma}
  For $x\in H$, $Q_x$ is nonempty if and only if $x\in C$. 
\end{lemma}

\begin{lemma}
  \label{lemma-polytope}
  For $x\in C$, $Q_x$ is a compact convex polytope. If $x\in C^\circ$, then
  \begin{equation*}
    \op{dim}Q_x=N-d.
  \end{equation*}
\end{lemma}
\begin{proof}
  Since $x\in C\subset H$, every $\lambda\in Q_x$ satisfies
  \begin{equation*}
    \sum_i\lambda_i=\varphi(A\lambda)=\varphi(x)=1.
  \end{equation*}
  Thus $Q_x$ is a closed subset of the standard simplex $\Delta^{N-1}$, and hence is compact. It is a convex polytope since it is defined by the affine equations $A\lambda=x$ and the finitely many linear inequalities $\lambda_i\geq0$.

  Now suppose that $x\in C^\circ$. To determine the dimension of $Q_x$, we construct a point $\lambda^0\in Q_x\cap \R^N_{>0}$. For each $i\in[N]$, since $x$ lies in the relative interior of $C$ in $H$, we may choose $\epsilon_i>0$ sufficiently small such that
  \begin{equation*}
    y_i\coloneqq\frac{x-\epsilon_i a_i}{1-\epsilon_i}\in C,
  \end{equation*}
  hence $x=\epsilon_i a_i+(1-\epsilon_i) y_i$. Expressing each $y_i$ as a convex combination of the $a_j$ and averaging these $N$ expressions for $x$, we obtain a point $\lambda^0\in Q_x$ such that $\lambda^0_i>0$ for every $i$.

  The affine space $A^{-1}(x)$ has dimension $N-d$, since $A$ is surjective. Because all coordinates of $\lambda^0$ are positive, a sufficiently small neighborhood of $\lambda^0$ in $A^{-1}(x)$ is contained in $\R^N_{\geq0}$. Therefore, $Q_x$ has nonempty relative interior in $A^{-1}(x)$, and consequently
  \begin{equation*}
    \op{dim}Q_x= \op{dim}A^{-1}(x)=N-d.
  \end{equation*}
\end{proof}

For $\lambda\in \R^N_{\geq0}$, we write
\begin{equation*}
  \op{supp}_+(\lambda)=\{i\,\,|\,\,\lambda_i>0\}.
\end{equation*}

\begin{lemma}
  \label{lemma-vertex-indep}
  A point $\lambda\in Q_x$ is a vertex of $Q_x$ if and only if the vectors $\{a_i\,|\,i\in \op{supp}_+(\lambda)\}$ are linearly independent.
\end{lemma}
\begin{proof}
  Suppose $\{a_i\,|\,i\in \op{supp}_+(\lambda)\}$ are linearly independent and $\lambda=(\mu+\nu)/2$ for some $\mu,\nu\in Q_x$. If $i\notin \op{supp}_+(\lambda)$, then $0=\lambda_i=(\mu_i+\nu_i)/2$, which implies $\mu_i=\nu_i=0$ since $\mu_i,\nu_i\geq0$.
  We also have $x=A\mu=A\nu$. Therefore, $\mu=\nu$ since $\mu,\nu$ are also supported on $\op{supp}_+(\lambda)$.

  Conversely, if $\{a_i\,|\,i\in \op{supp}_+(\lambda)\}$ are dependent, we can choose a nonzero $v\in\op{ker}A$ supported on $\op{supp}_+(\lambda)$. For sufficiently small $\epsilon>0$, $\lambda+\epsilon v$ and $\lambda-\epsilon v$ both belong to $Q_x$, hence $\lambda$ cannot be a vertex.
\end{proof}

\begin{lemma}
  \label{lemma-vertex-basis}
  For each $x\in C_{\mathrm{generic}}$, there is a bijection between the vertices $\op{Vert}(Q_x)$ and the set of labeled bases $B$ such that $x\in C_B$. Moreover, the vertex corresponding to $B$, denoted by $\lambda^B$, has support exactly $B$ and is given by
  \begin{equation}
  \label{eq-vertex-basis}
    \lambda^B_B=A^{-1}_Bx,\qquad \lambda^B_{[N]\backslash B}=0.
  \end{equation}
\end{lemma}
\begin{proof}
  Let $\lambda\in\op{Vert}(Q_x)$. By Lemma \ref{lemma-vertex-indep}, the vectors $\{a_i\,|\,i\in \op{supp}_+(\lambda)\}$ are linearly independent, so $|\op{supp}_+(\lambda)|\leq d$. On the other hand,
  \begin{equation*}
    x=A\lambda=\sum_{i\in \op{supp}_+(\lambda)}\lambda_i a_i,
    \qquad \sum_{i\in \op{supp}_+(\lambda)}\lambda_i=1,
  \end{equation*}
  hence $x$ lies in the convex hull of the $a_i$ with $i\in\op{supp}_+(\lambda)$. The definition of $C_{\mathrm{generic}}$ rules out $|\op{supp}_+(\lambda)|<d$. Therefore $|\op{supp}_+(\lambda)|=d$, so $\op{supp}_+(\lambda)$ is a labeled
  basis and $x\in C_{\op{supp}_+(\lambda)}$.

  Conversely, let $B$ be a labeled basis such that $x\in C_B$. Since $x$ is a convex combination of $\{a_b\,|\,b\in B\}$, the vector
  \begin{equation*}
    \lambda^B_B=A_B^{-1}x,\qquad
    \lambda^B_{[N]\backslash B}=0
  \end{equation*}
  belongs to $Q_x$. In fact, every coordinate of $\lambda^B_B$ is positive: otherwise $x$ would lie in the convex hull of a proper subset of $B$, which has fewer than $d$ elements and contradicts $x\in C_{\mathrm{generic}}$.
  Therefore, $\op{supp}_+(\lambda^B)=B$.
  Since the vectors $\{a_b\,|\,b\in B\}$ are linearly independent, $\lambda^B$ is a vertex of $Q_x$ by Lemma \ref{lemma-vertex-indep}.

  These two directions are inverse to each other since a vertex is determined uniquely by its support $B$ and the equation $A_B\lambda_B=x$. This proves the claimed bijection and \eqref{eq-vertex-basis}.
\end{proof}

For a vertex $v$ of a polytope $Q$, we define the tangent cone by
\begin{equation}
  T_v Q\coloneqq \op{cone}(Q,v)=\{\alpha(y-v)\,\,|\,\, \alpha\geq0,\, y\in Q\}.
\end{equation}

\begin{proposition}
  \label{prop-simple}
  For $x\in C_{\mathrm{generic}}$, let $\lambda^B$ be the vertex corresponding to $B$ and denote $J=[N]\backslash B$. We have
  \begin{equation}
    T_{\lambda^B}Q_x=\{u\in \op{ker}A\,\,|\,\, u_j\geq0 \text{ for every } j\in J\},
  \end{equation}
  where the projection
  \begin{equation*}
    \pi_J: \op{ker}A\to \R^J,\quad u\mapsto u_J
  \end{equation*}
  is an isomorphism.
  Under the coordinate identification $\R^N=\R^B\oplus\R^J$, the inverse is given by
  \begin{equation}
    \label{eq-inverse}
    \pi^{-1}_J: \R^J\to \op{ker}A,\quad z\mapsto (-A_B^{-1}A_J z,z).
  \end{equation}
  Therefore, $T_{\lambda^B}Q_x\simeq \R^J_{\geq0}$ and $Q_x$ is a simple $(N-d)$-polytope.
\end{proposition}
\begin{proof}
  If $u\in T_{\lambda^B} Q_x$, we can write $u=\alpha(y-\lambda^B)$ with $\alpha\geq0$ and $y\in Q_x$. Then $Au=0$ and $u_j=\alpha y_j\geq0$ for each $j\in J$ since $\lambda^B_j=0$.

  Conversely, suppose $Au=0$ and $u_j\geq0$ for each $j\in J$. Then $\lambda^B_b>0$ for each $b\in B$ by Lemma \ref{lemma-vertex-basis}, so for sufficiently small $\epsilon>0$ one has $\lambda^B_j+\epsilon u_j\geq 0$ for all $j\in [N]$.
  Since $A(\lambda^B+\epsilon u)=x$, we have $\lambda^B+\epsilon u\in Q_x$. This implies $u\in T_{\lambda^B}Q_x$.

  The map $\pi_J$ is injective: if $u\in\op{ker}\pi_J\subset \op{ker}A$, then $u$ is supported on $B$ and must vanish.
  The formula \eqref{eq-inverse} of $\pi_J^{-1}$ is a right-inverse of $\pi_J$ by definition. Therefore, $\pi_J$ is an isomorphism.
  As a result, $T_{\lambda^B}Q_x$ has exactly $N-d$ edge rays. Together with Lemma \ref{lemma-polytope}, we see $Q_x$ is a simple $(N-d)$-polytope.
\end{proof}

\begin{corollary}
  \label{cor-ext}
  Under the hypotheses and notation of Proposition \ref{prop-simple}, the edge ray of $Q_x$ indexed by $j\in J$ is generated by $c^{B,j}$. The functional $\lambda\mapsto\langle\rho,\lambda\rangle$ increases along this ray exactly when $j\in \op{Ext}_\rho(B)$.
\end{corollary}
\begin{proof}
  By \eqref{eq-inverse}, for the basis vector $e_j\in \R^J$, we have
  \begin{equation*}
    (\pi_J^{-1}(e_j))_j=1,\quad (\pi_J^{-1}(e_j))_B=-A_B^{-1}a_j,\quad (\pi_J^{-1}(e_j))_k=0\text{  for  }k\notin B\sqcup\{j\}.
  \end{equation*}
  Thus $\pi_J^{-1}(e_j)=c^{B,j}$. The remaining statement follows from Lemma \ref{lemma-sign}.
\end{proof}

\section{The fiberwise \texorpdfstring{$h$}{h}-polynomial}
\label{sec-fiber-h}

We review the construction and basic properties of the $h$-polynomial. We refer the reader to \cite[Chapter~8]{Zie} for more details.

For a simple $r$-polytope $Q$, the $h$-polynomial is defined as
\begin{equation}
  h(Q;t)\coloneqq \sum_{\emptyset\neq F\leq Q}(t-1)^{\op{dim}F},
\end{equation}
where $F\leq Q$ means $F$ is a face of $Q$. Note that the $h$-polynomial of a point equals $1$. 

\begin{lemma}
  \label{lemma-minimizer}
  Let $Q$ be a polytope and $\ell$ be a linear functional on its ambient space that is nonconstant on every edge of $Q$. Then every nonempty face of $Q$ has a unique vertex minimizing $\ell$.
\end{lemma}
\begin{proof}
  Since $\ell$ is linear, the set of minimizers on a face is also a face of $Q$. If it has positive dimension, then it contains an edge on which $\ell$ is constant, which leads to a contradiction.
\end{proof}

We now use a different expression of $h(Q;t)$, which is more convenient for the comparison with $f_A(t)$.

\begin{lemma}
  \label{lemma-h}
  Let $Q$ be a simple $r$-polytope and $\ell$ be nonconstant on each edge. Then we have
  \begin{equation}
    h(Q;t)=\sum_{v\in \op{Vert}(Q)}t^{w(v)},
  \end{equation}
  where $w(v)$ is the number of edges along which $\ell$ increases from $v$.
\end{lemma}
\begin{proof}
  For each face of $Q$, the minimizer of $\ell$ is exactly a vertex $v$ by Lemma \ref{lemma-minimizer}. 
  Conversely, for a vertex $v$, the faces containing $v$ correspond to faces of $T_vQ$. Since $Q$ is simple, every $i$-face of $T_vQ$ is generated by $i$ edge rays. The corresponding face of $Q$ has minimum at $v$ exactly when $\ell$ increases along all these rays. Therefore, the number of $i$-faces with minimum at $v$ is $\binom{w(v)}{i}$. Counting each $i$-face at its unique minimum gives
  \begin{equation}
    f_i(Q)=\sum_{v\in \op{Vert}(Q)}\binom{w(v)}{i},
  \end{equation}
  where $f_i(Q)$ denotes the number of $i$-faces of $Q$. Therefore,
  \begin{gather*}
    h(Q;t)=\sum_{\emptyset\neq F\leq Q}(t-1)^{\op{dim}F}=\sum_{i=0}^r f_i(Q)(t-1)^i\\ 
    =\sum_{v\in \op{Vert}(Q)}\sum_{i=0}^r \binom{w(v)}{i}(t-1)^i=\sum_{v\in \op{Vert}(Q)} t^{w(v)}.
  \end{gather*}
\end{proof}

We define on $H=\varphi^{-1}(1)$ the function
\begin{equation}
  \tilde{h}_{A,\rho}(x;t)\coloneqq \sum_{B}{1_{C_B}(x)}t^{\op{ext}_\rho(B)},
\end{equation}
where $1_{C_B}(x)$ is the indicator function of $C_B$.

\begin{proposition}
  \label{prop-fiber-h}
  For each $x\in C_{\mathrm{generic}}$, we have
  \begin{equation}
    h(Q_x;t)=\tilde{h}_{A,\rho}(x;t).
  \end{equation}
\end{proposition}
\begin{proof}
  Lemma \ref{lemma-vertex-basis} identifies the vertices of $Q_x$ with the labeled bases $B$ for which $x\in C_B$.
  Proposition \ref{prop-simple} says $Q_x$ is a simple polytope.
  Since $\rho$ is circuit-generic, Corollary \ref{cor-ext} shows that $\lambda\mapsto\langle\rho,\lambda\rangle$ is nonconstant on every edge and has exactly $\op{ext}_\rho(B)$ increasing edges at $\lambda^B$.
  The result now follows from Lemma \ref{lemma-h}.
\end{proof}

\section{Barycentric-fiber integration and trapezoidality}
\label{sec-integration}

In this section, we prove that $f_A(t)$ equals the integral of the $h$-polynomials of the generic barycentric fibers over the full-measure locus $C_{\mathrm{generic}}$. We then derive the trapezoidality of $f_A(t)$ from the classical $g$-theorem for polytopes.

First we define a measure on $H=\varphi^{-1}(1)$ by the volume of its truncated cone. For a Borel measurable set $S\subset H$, let
\begin{equation}
  \mu(S)\coloneqq \op{Vol}(\{sx,\,x\in S,\, 0\leq s\leq1\}).
\end{equation}
To see that this is a Borel measure, choose $x_0\in H$ and identify $H=x_0+\op{ker}\varphi$. In these coordinates, the Jacobian of
\begin{equation*}
  (0,1]\times H\longrightarrow V,\qquad (s,x)\longmapsto sx,
\end{equation*}
is a positive constant times $s^{d-1}$. This map is a bijection onto $\{y\in V \,\,|\,\, 0<\varphi(y)\leq1\}$, with inverse $y\mapsto(\varphi(y),y/\varphi(y))$. Consequently, $\mu$ is a positive constant multiple of the $(d-1)$-dimensional Lebesgue measure on $H$. In particular,
\begin{equation}
  \label{eq-generic-full-measure}
  \mu(C\backslash C_{\mathrm{generic}})=0,
\end{equation}
because this complement is contained in the relative boundary of $C$ together with finitely many convex hulls of fewer than $d$ points.

\begin{lemma}
  \label{lemma-volume}
  The volume of $\Pi_B$ and $\mu(C_B)$ are related by
  \begin{equation}
    \op{Vol}\Pi_B=d! \mu(C_B).
  \end{equation}
\end{lemma}
\begin{proof}
  The truncated cone over $C_B$ is the simplex
  \begin{equation*}
    \op{conv}(\{0\}\cup\{a_b \,\,|\,\, b\in B\}).
  \end{equation*}
  Its volume is $|\det A_B|/d!$, whereas the volume of the parallelotope $\Pi_B$ is $|\det A_B|$.
\end{proof}

For any circuit-generic $\rho$, we can now reinterpret $f_{A,\rho}(t)$:
\begin{gather*}
  f_{A,\rho}(t)=\sum_{B}t^{\op{ext}_{\rho}(B)}\op{Vol}\Pi_B = \sum_{B}t^{\op{ext}_{\rho}(B)} d!\mu(C_B)\\
  = d! \sum_{B}\int_C t^{\op{ext}_{\rho}(B)} 1_{C_B}(x) d\mu(x) = d! \int_C \big(\sum_{B}t^{\op{ext}_{\rho}(B)} 1_{C_B}(x)\big) d\mu(x)\\
  = d! \int_C \tilde{h}_{A,\rho}(x;t) d\mu(x) = d! \int_{C_{\mathrm{generic}}} h(Q_x;t) d\mu(x).
\end{gather*}
Here the last equality follows from Proposition \ref{prop-fiber-h} and \eqref{eq-generic-full-measure}. Since $C\setminus C_{\mathrm{generic}}$ is $\mu$-null, the values on nongeneric, possibly nonsimple fibers can be omitted.
Therefore, we have

\begin{proposition}
  \label{prop-bary-formula}
  For every circuit-generic $\rho$, one has the barycentric-fiber integration formula
  \begin{equation}
    f_{A,\rho}(t)=d! \int_{C_{\mathrm{generic}}} h(Q_x;t) d\mu(x).
  \end{equation}
  In particular, $f_{A,\rho}(t)$ is independent of $\rho$ and equals $f_A(t)$.
\end{proposition}
\begin{proof}
  The formula was established above. Its right-hand side depends only on the fibers $Q_x$, so it is independent of $\rho$.
\end{proof}

It remains to use the trapezoidal property of each polynomial $h(Q_x;t)$. This follows from the necessity direction of the classical $g$-theorem \cite[Theorem~8.35]{Zie}, applied to the simplicial polytope dual to $Q_x$. This direction was proved by Stanley \cite{StanleyGTheorem}; the characterization of $M$-sequences used in its statement goes back to Macaulay \cite{Mac27}. We record only the part that we need:

\begin{theorem}
  \label{thm-g-theorem}
  Let $Q$ be a simple $r$-polytope and
  \begin{equation}
    h(Q;t)=\sum_{i=0}^r h_i(Q)t^i.
  \end{equation}
  Then 
  \begin{equation*}
      h_0(Q)=1,\quad h_i(Q)=h_{r-i}(Q)\quad (0\leq i\leq r).
  \end{equation*}
  Moreover, the tuple $(g_0(Q),\dots,g_m(Q))$, $m=\lfloor r/2\rfloor$ defined by
  \begin{equation*}
    g_0(Q)=1,\quad g_i(Q)=h_i(Q)-h_{i-1}(Q), \,1\leq i\leq m
  \end{equation*}
  is an $M$-sequence. Consequently,
  \begin{equation}
    g_i(Q)\geq 0,\quad g_i(Q)=0 \Longrightarrow g_j(Q)=0\text{  for  } i\leq j\leq m.
  \end{equation}
\end{theorem}
An $M$-sequence is a nonnegative integer sequence satisfying Macaulay's growth inequalities. If one term is zero, these inequalities force all subsequent terms to vanish; this proves the final implication.

Now we have enough ingredients to prove the main theorem:

\begin{theorem}
  \label{thm-main}
  Let $A$ be a full-dimensional flat vector arrangement. The coefficient sequence of $f_A(t)$ is trapezoidal.
\end{theorem}
\begin{proof}
  Let $r=N-d$.
  Since $\op{Ext}_\rho(B)\subset[N]\backslash B$ for every labeled basis $B$, one has $\deg f_A\leq r$.
  We can write
  \begin{equation*}
    f_A(t)=\sum_{i=0}^{r}a_i t^i,\quad h(Q_x;t)=\sum_{i=0}^r h_i(x)t^i,
  \end{equation*}
  where $h_i(x)$ is defined on $C_{\mathrm{generic}}$. Denote $m=\lfloor r/2\rfloor$.
  By Proposition \ref{prop-bary-formula}, we have
  \begin{equation*}
    a_i=d!\int_{C_{\mathrm{generic}}} h_i(x)d\mu(x).
  \end{equation*}
  
  By Theorem \ref{thm-g-theorem}, $h_i(x)=h_{r-i}(x)$, which gives
  \begin{equation}
    \label{eq-symmetry}
    a_i=a_{r-i},\quad i=0,\dots,m.
  \end{equation}
  
  For each $1\leq i\leq m$, 
  \begin{equation}
    \label{eq-monotone}
    a_i-a_{i-1}=d! \int_{C_{\mathrm{generic}}} (h_i(x)-h_{i-1}(x))d\mu(x)\geq0.
  \end{equation}
  The functions $h_i$ are measurable because, on $C_{\mathrm{generic}}$, they are coefficients of the finite sum $\tilde h_{A,\rho}$. Suppose $a_i-a_{i-1}=0$ for some $1\leq i\leq m$. Since the integrand in \eqref{eq-monotone} is nonnegative, $h_i(x)-h_{i-1}(x)=0$ for almost every $x\in C_{\mathrm{generic}}$.
  Therefore, Theorem \ref{thm-g-theorem} implies $h_j(x)-h_{j-1}(x)=0$ for almost every $x\in C_{\mathrm{generic}}$ for every $i\leq j\leq m$, and hence
  \begin{equation}
    \label{eq-plateau}
    a_j-a_{j-1}=0\text{  for  }i\leq j\leq m.
  \end{equation} 
  
  Finally, since $h_0(x)=1$ for $x\in C_{\mathrm{generic}}$,
  \begin{equation}
    \label{eq-positive}
    a_0=d! \mu(C_{\mathrm{generic}})=d! \mu(C)>0,
  \end{equation}
  where the last inequality holds because $C$ is full-dimensional in $H$. It follows from \eqref{eq-symmetry} and \eqref{eq-monotone} that $a_i>0$ for all $i$.
  
  All conditions of Definition \ref{def-trapezoidal} are now fulfilled. 
\end{proof}

\section{Eulerian digraphs and the Murasugi--Stoimenow polynomial}
\label{sec-eulerian-digraphs}

In this section, we apply Theorem \ref{thm-main} to the Murasugi–Stoimenow polynomial and prove Corollary \ref{cor-euleriandiagraph-trapezoidal}. The bridge to flat arrangements is the cographic realization of Kálmán, Mészáros, and Postnikov \cite[Sections~4.2 and~5.1]{KMP}.

We first recall some standard terminology from graph theory.
A \emph{digraph} $D$ is a finite directed multigraph with vertex set $V(D)$ and edge set $E(D)$; parallel edges and loops are allowed. Its \emph{underlying graph} is obtained by forgetting the edge directions while retaining edge multiplicities, and $D$ is \emph{connected} when this underlying graph is connected. For $v\in V(D)$, the out-degree $\deg_D^+(v)$ and in-degree $\deg_D^-(v)$ count outgoing and incoming edge-ends, respectively; a loop contributes one to each.
A connected digraph $D$ is \emph{Eulerian} if
\begin{equation*}
  \deg_D^+(v)=\deg_D^-(v)\qquad (v\in V(D)).
\end{equation*}

A \emph{spanning tree of $D$} means a spanning tree of its underlying graph; parallel edges remain distinct choices. Fix a root vertex $v_0\in V(D)$ and root such a tree $T$ at $v_0$. Each edge of $T$ then joins a vertex to its parent. The edge points \emph{toward} $v_0$ when it is directed from the child to the parent, and \emph{away from} $v_0$ otherwise. An \emph{oriented spanning tree rooted at $v_0$} is a spanning tree in which every edge points toward $v_0$; equivalently, for every vertex $v$, the unique path in $T$ from $v$ to $v_0$ is directed toward the root. For an arbitrary spanning tree $T$, define
\begin{equation*}
  \kappa_{v_0}(T)\coloneqq \#\{e\in E(T)\,\,|\,\, e\text{ points away from $v_0$ in the rooted tree $T$}\}.
\end{equation*}
Following the terminology in \cite[Section~2.3]{KMP}, we call $T$ a
\emph{$k$-spanning tree rooted at $v_0$} when $\kappa_{v_0}(T)=k$.
Reversing precisely these $k$ edges turns $T$ into an oriented spanning tree rooted at $v_0$.

Loops do not lie in a spanning tree, and deleting them preserves the Eulerian balance condition and the polynomial defined below. Therefore, we suppress loops, continue to denote the resulting loopless digraph by $D$, and let $n=|V(D)|$ and $m=|E(D)|$.

\begin{definition}
  \label{def-Murasugi--Stoimenow}
  For $0\leq k\leq n-1$, let
  \begin{equation*}
    c_k(D,v_0)\coloneqq
    \#\{T \,\,|\,\, T\text{ is a $k$-spanning tree of $D$ rooted at $v_0$}\}.
  \end{equation*}
  The \emph{Murasugi--Stoimenow polynomial} of the rooted Eulerian digraph $(D,v_0)$ is
  \begin{equation}
    \label{eq-Murasugi--Stoimenow}
    P_{D,v_0}(t)\coloneqq \sum_{k=0}^{n-1}c_k(D,v_0)t^k.
  \end{equation}
\end{definition}

Murasugi and Stoimenow proved that $c_k(D,v_0)$ is independent of the root,
\begin{equation}
  \label{eq-ms-palindromicity}
  c_k(D,v_0)=c_{n-1-k}(D,v_0)\quad\text{  for}\quad 0\leq k\leq n-1,
\end{equation}
and that every coefficient in this range is positive \cite[Proposition~1(1)-(4)]{MS03}. We therefore denote these simply as $c_k(D)$ and $P_D(t)$. 

We also recall the planar terminology used in the knot-theoretic application. A \emph{plane graph} is a graph equipped with a fixed embedding in the oriented plane; in particular, the incident edge-ends at each vertex have a cyclic order. A \emph{dimap}, or directed map, is a plane graph whose edges are directed. An \emph{alternating dimap} is a plane Eulerian digraph for which incoming and outgoing edge-ends alternate in the cyclic order around every vertex \cite[Section~2.3]{KMP}. If $G$ is a plane bipartite graph, orient each edge of its planar dual $D=G^*$ so that a fixed color class of $G$ lies on the right of the dual edge. This orientation makes $D$ an alternating dimap. For the corresponding special alternating link $L_G$, Murasugi and Stoimenow proved
\begin{equation*}
  P_D(t)\doteq \Delta_{L_G}(-t),
\end{equation*}
where $\doteq$ denotes equality up to multiplication by $\pm t^j$ \cite[Theorem~2]{MS03}.

Next, we record the cut-balance consequence of the Eulerian condition.

\begin{lemma}
  \label{lemma-eulerian-cut-balance}
  Let $D$ be Eulerian and let $V(D)=S\sqcup S^c$. Then the number of edges directed from $S$ to $S^c$ equals the number directed from $S^c$ to $S$.
\end{lemma}
\begin{proof}
  We take the sum of $\deg_D^+(v)-\deg_D^-(v)=0$ over all $v\in S$. Each edge with both endpoints in $S$ contributes once with each sign and cancels. The remaining terms count the edges from $S$ to $S^c$ positively and those from $S^c$ to $S$ negatively, so their difference is zero.
\end{proof}

Let us also recall the cographic matrix of \cite[Definition~4.8]{KMP}.
Choose a spanning tree $T_0$ of the underlying graph and let
\begin{equation*}
  F\coloneqq E(D)\backslash E(T_0),\qquad \beta\coloneqq |F|=m-n+1.
\end{equation*}
We write $F=\{e_1,\dots,e_\beta\}$ and let $\mathbf{e}_1,\dots,\mathbf{e}_\beta$ be the standard basis of $\R^\beta$. For every cotree edge $e_i\in F$, define $b_{e_i}=\mathbf{e}_i$.

For a tree edge $a\in E(T_0)$, deleting $a$ divides $T_0$ into vertex sets $S_a$ and $S_a^c$. For $1\leq i\leq \beta$, define
\begin{equation*}
  \varepsilon_{a,i}\coloneqq
  \begin{cases}
    0, & e_i\text{ does not cross the cut }S_a\,|\, S_a^c,\\
    1, & e_i\text{ crosses the cut in the direction opposite to }a,\\
    -1, & e_i\text{ crosses the cut in the same direction as }a,
  \end{cases}
\end{equation*}
and set
\begin{equation}
  \label{eq-cographic-column}
  b_a\coloneqq \sum_{i=1}^\beta\varepsilon_{a,i}\mathbf e_i.
\end{equation}
Let $\{\delta_e \,\,|\,\, e\in E(D)\}$ denote the standard basis of $\R^{E(D)}$. The resulting labeled matrix
\begin{equation}
  \label{eq-cographic-matrix}
  B_{D,T_0}:\R^{E(D)}\to\R^\beta, \qquad B_{D,T_0}(\delta_e)=b_e,
\end{equation}
has the block form $[K \,|\, I_\beta]$. 
The following proposition collects the properties of the cographic matrix that we need from \cite{KMP}.

\begin{proposition}
  \label{prop-flat-cographic}
  Let $D$ be a connected Eulerian digraph with at least two vertices, and let $T_0$ be a spanning tree. Then the columns of $B_{D,T_0}$ have the following properties.
  \begin{enumerate}
    \item They form a full-dimensional flat arrangement. More precisely, for
    \begin{equation*}
      \varphi_D(x_1,\dots,x_\beta)\coloneqq x_1+\cdots+x_\beta,
    \end{equation*}
    we have $\varphi_D(b_e)=1$ for every $e\in E(D)$.
    
    \item The matrix $B_{D,T_0}$ is totally unimodular and represents the cographic matroid of the underlying graph, with the orientation determined by $D$.
    
    \item With the standard volume form on $\R^\beta$,
    \begin{equation}
      \label{eq-pd-equals-fa}
      P_D(t)=f_{B_{D,T_0}}(t).
    \end{equation}
    In particular, the corank of this arrangement is
    \begin{equation}
      \label{eq-cographic-corank}
      |E(D)|-\op{rank}B_{D,T_0}=m-\beta=n-1.
    \end{equation}
  \end{enumerate}
\end{proposition}
\begin{proof}
  The identity block in \eqref{eq-cographic-matrix} shows that the columns span $\R^\beta$. It also gives $\varphi_D(b_{e_i})=1$ for every $e_i\in F$.

  Let $a\in E(T_0)$, and label the two components of $T_0-a$ by $S_a$ and $S_a^c$ so that $a$ is directed from $S_a$ to $S_a^c$. Let $P_a$ be the set of cotree edges crossing this cut from $S_a$ to $S_a^c$, and let $Q_a$ be the set of cotree edges crossing it from $S_a^c$ to $S_a$. Write
  \[
  p=|P_a|, \qquad q=|Q_a|.
  \]
  No edge of $T_0\setminus\{a\}$ crosses the cut, because $S_a$ and $S_a^c$ are the two components of $T_0-a$. Hence the complete set of edges directed from $S_a$ to $S_a^c$ is $\{a\}\sqcup P_a$, whereas the complete set of edges directed from $S_a^c$ to $S_a$ is $Q_a$. Lemma~\ref{lemma-eulerian-cut-balance} therefore gives
  \[
  1+p=q.
  \]
  By the definition of the signs $\varepsilon_{a,i}$, it follows that
  \[
  \varphi_D(b_a) = \sum_{i=1}^{\beta}\varepsilon_{a,i}
  = q-p = 1.
  \]
  Thus every column of $B_{D,T_0}$ lies in $\varphi_D^{-1}(1)$. This is the argument by \cite[Lemma~4.12]{KMP}.

  Total unimodularity and the representation of the underlying cographic matroid follow from \cite[Theorem~4.9]{KMP}. The fact that the represented oriented matroid is the oriented cographic matroid determined by $D$, independently of the auxiliary spanning tree $T_0$, is recorded in \cite[Remark~4.13]{KMP}. Consequently, every nonzero maximal minor of $B_{D,T_0}$ is equal to $\pm1$, so every basis parallelotope has volume one with respect to the standard volume form on $\mathbb R^\beta$.

  Fix the root $v_0$ used in the definition of $P_{D,v_0}(t)$. Choose an Eulerian tour of $D$ beginning at $v_0$, and write
  \[
  d_1\prec_\sigma d_2\prec_\sigma\cdots\prec_\sigma d_m
  \]
  for the edges in their order of traversal. Thus the tail of $d_1$ is $v_0$. Choose $0<\eta\ll1$ and define
  \[
  \rho_\sigma := \sum_{j=1}^{m}\eta^j\delta_{d_j} \in \mathbb R^{E(D)}.
  \]
  If $c$ is a circuit vector and $d_j$ is the $\sigma$-smallest edge in $\operatorname{supp}(c)$, then
  \[
  \langle\rho_\sigma,c\rangle = \eta^j \left(c_{d_j}+\sum_{\ell>j}c_{d_\ell}\eta^{\ell-j}\right).
  \]
  For sufficiently small $\eta$, this scalar is nonzero and has the same sign as $c_{d_j}$. For each circuit support, choose one nonzero circuit vector. Since there are only finitely many circuit supports, $\eta>0$ may be chosen sufficiently small, uniformly over these representatives, so that $\langle\rho_\sigma,c\rangle$ has the same sign as the coefficient of the $\sigma$-smallest edge of $c$. The same conclusion then holds for every nonzero scalar multiple. Thus $\rho_\sigma$ is circuit-generic and induces the order orientation described in \cite[Remark 3.2]{KMP}.

  For each spanning tree $T$ of the underlying graph, set
  \[
  B_T:=E(D)\setminus E(T).
  \]
  This is a basis of the cographic matroid represented by $B_{D,T_0}$. If $d\in E(T)$, then the unique cographic circuit contained in $B_T\cup\{d\}$ is the fundamental cut $C^*(T,d)$ determined by the two components of $T-d$. Let $S_d$ be the vertex set of the component containing $v_0$. The $\sigma$-smallest edge of $C^*(T,d)$ is the first edge of the Eulerian tour that crosses the cut $S_d\mid S_d^c$. Since the tour begins in $S_d$, this first crossing is directed from $S_d$ to $S_d^c$.

  It follows from \cite[Lemma~5.1]{KMP} that $d$ is externally semi-active for the basis $B_T$ with respect to $\rho_\sigma$ if and only if $d$ also points from $S_d$ to $S_d^c$. Equivalently, $d$ points away from $v_0$ in the rooted tree $T$. Therefore,
  \[
  \operatorname{ext}_{\rho_\sigma}(B_T)=\kappa_{v_0}(T)
  \]
  for every spanning tree $T$, which is \cite[Lemma~5.2]{KMP}. Since all basis parallelotopes have volume one, we obtain
  \[
  \begin{aligned}
  f_{B_{D,T_0}}(t)
  &= f_{B_{D,T_0},\rho_\sigma}(t)\\
  &= \sum_T t^{\operatorname{ext}_{\rho_\sigma}(B_T)}\\
  &= \sum_T t^{\kappa_{v_0}(T)}\\
  &= P_{D,v_0}(t)=P_D(t).
  \end{aligned}
  \]
  This proves \eqref{eq-pd-equals-fa}; equivalently, it is the cographic identity of \cite[Theorem~5.3]{KMP}. Finally, the identity block in $B_{D,T_0}$ gives
  \[
  \operatorname{rank}B_{D,T_0}=\beta.
  \]
  Since $\beta=m-n+1$, we have
  \[
  |E(D)|-\operatorname{rank}B_{D,T_0} = m-\beta = n-1,
  \]
  which proves \eqref{eq-cographic-corank}.
\end{proof}

Now we can prove the main result of this section:

\begin{corollary}
  \label{cor-eulerian-digraph-trapezoidality}
  Let $D$ be a connected Eulerian digraph on $n$ vertices, and write
  \begin{equation*}
    P_D(t)=\sum_{k=0}^{n-1}c_k(D)t^k.
  \end{equation*}
  The coefficient sequence is trapezoidal. More explicitly, with $q=\lfloor(n-1)/2\rfloor$,
  \begin{equation}
    \label{eq-eulerian-trapezoidality}
    c_k(D)=c_{n-1-k}(D)>0, \qquad c_0(D)\leq c_1(D)\leq\cdots\leq c_q(D).
  \end{equation}
  If $c_i(D)=c_{i-1}(D)$ for some $1\leq i\leq q$, then
  \begin{equation*}
    c_j(D)=c_{j-1}(D)\quad \text{  for}\quad i\leq j\leq q.
  \end{equation*}
  Equivalently, the coefficients strictly increase until a single centered plateau, which may be the whole sequence, and then decrease symmetrically.
\end{corollary}
\begin{proof}
  The loopless one-vertex digraph has $P_D(t)=1$. Otherwise, Proposition \ref{prop-flat-cographic} gives a positive-rank full-dimensional flat arrangement $B_{D,T_0}$ with $P_D(t)=f_{B_{D,T_0}}(t)$. 
  Theorem \ref{thm-main} then applies with the coefficient sequence indexed through the corank. 
  Equation \eqref{eq-cographic-corank} identifies this corank with $n-1$, so the resulting coefficient statement is exactly \eqref{eq-eulerian-trapezoidality} and the asserted propagation of equality.
\end{proof}

\begin{remark}
  Corollary \ref{cor-eulerian-digraph-trapezoidality} extends the trapezoidality question of coefficients posed by Murasugi and Stoimenow from the planar alternating setting to every connected Eulerian digraph. Hafner, Mészáros, and Vidinas conjecture the stronger log-concavity inequalities
  \begin{equation*}
    c_k(D)^2\geq c_{k-1}(D)c_{k+1}(D)\quad\text{  for} \quad 1\leq k\leq n-2,
  \end{equation*}
  together with the absence of internal zeros \cite[Conjecture~1.6]{HMVGeneralized}. The latter condition already follows from strict positivity; the present argument establishes trapezoidality but not the nonlinear inequalities above.
\end{remark}

\printbibliography

\end{document}